\documentclass[12pt]{article}

\usepackage{times, amsmath, amsfonts, amssymb, amsthm}
\usepackage{times}
\usepackage{srcltx}
\usepackage{setspace}

\newcommand{\Norm}[1] {|\!|\!|#1|\!|\!|}
\newcommand{\BigNorm}[1] {\Big|\!\Big|\!\Big|#1\Big|\!\Big|\!\Big|}
\newcommand{\dd}{\text{\rm d}}             
\newcommand{\ee}{\text{\rm e}}

\newcommand{\cF}{\mathcal{F}}

\newcommand{\cL}{\mathcal{L}}
\newcommand{\cM}{\mathcal{M}}

\newcommand{\1}{{\bf{1}}}

\theoremstyle{plain}
\newtheorem{theorem}{Theorem}[section]
\newtheorem{lemma}[theorem]{Lemma}
\newtheorem{proposition}[theorem]{Proposition}
\newtheorem{corollary}[theorem]{Corollary}

\theoremstyle{definition}
\newtheorem{remark}[theorem]{Remark}

\renewenvironment{proof}[1][] {\noindent {\bf Proof#1.} }{\hspace*{\fill}$\square$\medskip\par}

\newenvironment{proofone}[1][] {\noindent {\bf Proof of Theorem \ref{main}#1.} }{\hspace*{\fill}$\square$\medskip\par}

\newcommand{\N}{{\mathbf N}}

\newcommand{\R}{{\mathbf R}}

\newcommand{\E}{{\mathbf E}}
\newcommand{\F}{{\cal F}}

\renewcommand{\P}{{\mathbf P}}

\newcommand{\rC}{{\mathrm C}}
\newcommand{\rL}{{\mathrm L}}
\newcommand{\rM}{{\mathrm M}}

\author{Michael Scheutzow
  \thanks{Institut f\"ur Mathematik, MA 7-5, Fakult\"at II, 
        Technische Universit\"at Berlin, 
        Stra\ss e des 17.~Juni 136, 10623 Berlin, FRG;   
        \small{\tt ms{\scriptsize @}math.tu-berlin.de}}}

\title{Exponential growth rate for a singular linear stochastic delay differential equation}

\date{\today}

\begin{document}  \maketitle

\begin{abstract}\noindent
We establish the existence of a deterministic exponential growth rate for the norm (on an appropriate function space) 
of the solution of the linear scalar stochastic delay equation $\dd X(t) = X(t-1)\,\dd W(t)$ which does not depend on the initial condition as long 
as it is not identically zero. Due to the singular nature of the equation this property does not follow from 
available results on stochastic delay differential equations. The key technique is to establish existence and uniqueness 
of an invariant measure of the projection of the solution onto the unit sphere in the chosen function space via 
{\em asymptotic coupling} and to prove a Furstenberg-Hasminskii-type formula (like in the finite dimensional case).

  \par\medskip

\end{abstract}


\section{Introduction}

Let $W(t),\, t \ge 0$ be linear Brownian motion defined on some probability space $(\Omega,\F,\P)$. In this paper, we will 
study asymptotic properties of the stochastic delay differential equation (SDDE)
\begin{equation}
\label{sde}
\dd X(t)=X(t-1)\,\dd W(t), \; X_0=\eta,
\end{equation}
where
$\eta \in \rC:=\rC ([-1,0],\R)$, and for $t \ge 0$, we define
$$
X_t(s):=X(t+s),\,s \in [-1,0].
$$  
For a fixed chosen norm $\|.\|$ on $\rC$ we will be interested in the question whether for the $\rC$-valued 
solution $X_t$ of the SDDE \eqref{sde} the limit
$$
\lambda(\eta,\omega):=\lim_{t \to \infty} \frac 1t \log \|X_t(\omega)\|
$$
exists almost surely for each $\eta \in \rC$ and is deterministic and independent of $\eta$ as long as $\eta \neq 0$. 
We will show in our main result (Theorem \ref{main}) that there exists a deterministic number $\Lambda \in \R$ such that for every $\eta \neq 0$ we have
$\lambda(\eta,\omega)=\Lambda$ almost surely. In this case, we call $\Lambda$ the exact exponential growth rate of 
\eqref{sde}. To prove Theorem \ref{main}, we  follow the path paved by Furstenberg \cite{F63} and Hasminskii \cite{H67} (see also \cite{AKO84}) 
in the finite dimensional case: 
project the solution of the equation to the unit sphere of an appropriate function space and show that the induced Markov process has a unique invariant 
probability measure $\mu$. Then make sure that for each initial condition on the sphere the empirical measure converges to $\mu$ and represent 
the exponential growth rate as an integral with respect to $\mu$ as in the classical Furstenberg formula.  While the existence of $\mu$ is rather 
easy to show, uniqueness is more involved.  We follow the strategy developed in \cite{HMS11} to show uniqueness of $\mu$. Contrary to \cite{HMS11} we have to 
deal with degenerate equations here requiring a modification of the approach.

Let us first justify our restriction to such a simple equation as \eqref{sde}. In spite of its simplicity, the equation is known to be 
{\em singular} in the sense that there does not exist any modification of the solution which almost surely  depends continuously 
upon the initial condition $\eta$ with respect to the sup-norm (see \cite{M86}). In particular, the results in \cite{MS96} establishing 
a {\em Lyapunov spectrum} and a corresponding decomposition of the state space for a large class of {\em regular} linear SDDEs cannot be applied. 

Since equation \eqref{sde} is the simplest possible singular stochastic delay equation, we believe that it is worthwhile studying its 
asymptotics in some detail. We are optimistic that in principle our method of proof can be generalized to a large class of 
(multidimensional) linear stochastic functional diffential equations but we expect the proofs to be quite a bit more technical.

Clearly, equation \eqref{sde} has a unique solution for each initial condition $X_0=\eta \in C$ and the process 
$X_t$, $t \ge 0$ is a (strong) $\rC -$valued Markov process 
with continuous paths.
We define the following norms on $\rC$:
Let $\|.\|_2$ be the $\rL_2$-norm, $\|.\|$ the sup-norm, and $\Norm{.}$ the $\rM_2$-norm defined as
$$
\Norm{f}^2:=(f(0))^2 + \int_{-1}^0 (f(s))^2\,\dd s 
$$
(the Hilbert space $\rM_2$ consists of all functions from $[-1,0]$ to $\R$ for which this norm is finite).

Our main result in this paper is the following:
\begin{theorem}\label{main}
There exists a number $\Lambda \in \R$ such that for each $\eta \in \rC\backslash\{0\}$, the solution $X$ of equation \eqref{sde} with initial condition 
$\eta$ satisfies 
$$
\lim_{t \to \infty} \frac 1t \log \|X_t(\omega)\| = \lim_{t \to \infty} \frac 1t \log \Norm{X_t(\omega)}=\Lambda\qquad \mbox{ a.s.}.
$$
\end{theorem}
It is easy to see (and will follow from Lemma \ref{lemma1}) that for each $\eta \neq 0$, the process $X_t$ starting at $X_0=\eta$ 
will almost surely never become (identically) zero. Therefore, the process
$$
S_t:=X_t/\Norm{X_t},\; t \ge 0
$$ 
is well-defined. Since our equation \eqref{sde} is linear, the process $S_t,\;t \ge 0$ is a Markov process with continuous paths 
(with respect to both the sup-norm and the $\rM_2$-norm on $\rC$) 
on the unit sphere of $\rM _2$. We will show that this process has a unique invariant probability measure $\mu$.
Suppose for a moment that this has been shown. Then, by It\^o's formula, we have
$$
\dd \Norm{X_t}^2 = X^2(t)\, \dd t + 2 X(t)X(t-1)\,\dd W(t)=\Norm{X_t}^2 \Big( f\Big( \frac {X_t}{ \Norm{X_t}} \Big) \dd t
+ g\Big( \frac {X_t}{ \Norm{X_t}} \Big) \dd W(t) \Big), 
$$
where $f(\eta)=\eta^2(0)$ and $g(\eta)=2\eta(0)\eta(-1)$. Hence,
$$
\dd(\log  \Norm{X_t}) = \Big( \frac 12 f\Big( \frac {X_t}{ \Norm{X_t}} \Big) 
-\frac 14  g^2\Big( \frac {X_t}{ \Norm{X_t}} \Big)\Big)\, \dd t +  \frac 12 g\Big( \frac {X_t}{ \Norm{X_t}} \Big)\,\dd W(t). 
$$
Therefore, by Birkhoff's ergodic theorem,
\begin{equation}\label{Grenzwert}
\lim_{t \to \infty} \frac 1t \log  \Norm{X_t}  = \int \frac 12 f(\eta)-\frac 14 g^2(\eta)\,\dd \mu(\eta)=:\Lambda
\mbox { a.s.}, 
\end{equation}
for $\mu$-almost every initial condition $X_0=\eta$ since $f$ is bounded (and $g^2$ is non-negative) and since the stochastic integral is asymptotically 
negligible compared to its quadratic variation unless the latter process remains bounded as $t \to \infty$ in which case the stochastic integral remains bounded in $t$ as well 
and therefore does not contribute towards the limit in \eqref{Grenzwert}.
This is almost everything we want to show except that we  want to ensure that the limit exists almost surely for {\em each} initial 
condition $\eta$ and not just for $\mu$-almost every $\eta$. 
Since $f$ is bounded, it follows that $\Lambda < \infty$ (in fact $\Lambda \le 1/2$) but it is not immediately obvious that $\Lambda > -\infty$. 
This follows however from the following result which is Theorem 2.3. in \cite{MS97}.
\begin{proposition}\label{aop}
There exists a real number $\Lambda_0$ such that for every $\eta \in  \rC\backslash\{0\}$, we have
$\P\{\liminf_{t \to \infty}\frac 1t \log \Norm{X_t} \ge \Lambda_0\}=1$, where $X$ solves \eqref{sde} with initial condition $\eta$.
\end{proposition}
Note that as a consequence Proposition \ref{aop}, it follows that the function $g$ is square integrable with respect to $\mu$.

It is easy to see that then, we also have 
$$
\lim_{t \to \infty} \frac 1t \log  \|X_t\|=\Lambda \mbox { a.s.},
$$
since 
$$
\Norm{X_t} \le \sqrt{2} \|X_t\| \le \sqrt{2} \sup_{t-1 \le s \le t} \Norm{X_s}
$$
for all $t \ge 0$.

In order to prove Theorem \ref{main}, it therefore remains to prove existence and uniqueness of an invariant probability measure $\mu$ 
of the Markov process $S_t,\, t\ge 0$ and to show that \eqref{Grenzwert} holds for {\em each} initial condition $\eta \in \rC \backslash \{0\}$.


We will need the following result which is Step 1 in the proof of Theorem 2.3 in \cite{MS97}.
\begin{proposition}\label{aop1}
There exists a real number $K$ such that for every $\eta \in  \rC$, we have
$\E \big(\Norm{X_1}^{-1/2}\big) \le K\Norm{\eta}^{-1/2}$, where $X$ solves \eqref{sde} with initial condition $\eta$.
\end{proposition}

Upper and lower bounds for the exponential growth rate $\Lambda$ have been obtained (even for equations with an additional factor $\sigma$ in front of $\dd W(t)$) 
in \cite{MS97} and \cite{Sch05} (in those papers the existence of the limit  \eqref{Grenzwert} was not yet known: the authors obtained upper deterministic 
bounds for the $\limsup$ and lower deterministic bounds for the $\liminf$).

\section{Existence of an invariant measure}

In this section, $X$ is always the solution of equation \eqref{sde} -- possibly with a random initial condition which is independent of 
the $\sigma-$algebra generated by the driving Wiener process $W$. Let $\F_t$, $t \ge 0$ be the filtration (right-continuous and complete) 
generated by the initial condition and the Wiener process $W$. We will always assume that the initial condition satisfies $\E\|X_0\|^2< \infty$ 
which ensures that all moments appearing below will be finite and conditional expectations well-defined. As before, we define $S_t:=X_t/\Norm{X_t}$. 
We need the following lemmas.

\begin{lemma} \label{lemma1}
There exists some $c_1>0$ such that for each $t \ge 1$ and $\alpha \ge 0$ we have
$$
 \P\{|X(t)|\le \alpha \Norm{X_{t-1}} \big| \F_{t-1}\} \le c_1 \alpha \mbox{ a.s.}.
$$
\end{lemma}
\begin{proof} Let $N$ be a standard normal random variable. Abbreviate $a:=X(t-1)$, $\sigma:=\|X_{t-1}\|_2$. Then, for $\alpha \le 1/2$,
\begin{align*}
\P\{&|X(t)|\le \alpha \Norm{X_{t-1}} \big| \F_{t-1}\}\\
&= \P\big\{\big| X(t-1)+\int_{t-1}^t X(s-1)\, \dd W(s) \big| \le \alpha \Norm{X_{t-1}} \big| \F_{t-1}\big\}\\
&\le\P\{|a+\sigma N| \le \alpha(|a|+\sigma)\big| \F_{t-1} \}\\
&\le \sup_{x\ge 0}\P\{N \in [x-\alpha(x+1),x+\alpha(x+1)]\}\\
&\le \sup_{x\ge 0}  \Big\{2\alpha (x+1) \frac 1{\sqrt{2\pi}}\exp\big\{-\frac 12 (((1-\alpha)x-\alpha)^+)^2\big\}\Big\}\\
&\le \alpha \sup_{x\ge 0} \Big\{2 (x+1) \frac 1{\sqrt{2\pi}}\exp\big\{-\frac 14 ((x-1)^+)^2\big\}\Big\}\\
&= c \alpha,
\end{align*}
since the supremum is finite. Defining $c_1:= c \vee 2$, the assertion follows.
\end{proof}

\begin{lemma} \label{lemma2}
For each $t \ge 1$ and each $\F_{t-1}$-measurable positive random variable    $\xi$, we have
$$
\P\{\|X_t\| \ge \xi \Big| \F_{t-1}\} \le \frac 1{\xi^2} \Big(X^2(t-1)+4\|X_{t-1}\|_2^2 \Big) \le \frac{4}{\xi^2}\Norm{X_{t-1}}^2\mbox{ a.s.}.
$$
\end{lemma} 
\begin{proof} Using Doob's $\rL^2$-martingale inequality, we get
$$
\P\{\|X_t\| \ge \xi \Big| \F_{t-1}\} \le  \frac 1{\xi^2} \E\big( \|X_t\|^2| F_{t-1}\big) \le \frac 1{\xi^2} \big( X^2(t-1)+4\|X_{t-1}\|_2^2\big) 
\le \frac{4}{\xi^2}\Norm{X_{t-1}}^2 \mbox{ a.s.}.
$$
\end{proof}

We regard the process $S_t,\,t \ge 0$ as a Markov process with state space $\bar \rC$ defined as the intersection of $\rC$ and the unit sphere of $\rM_2$ 
equipped with the supremum norm $\|.\|$. Then $S_t,t \ge 0$ is a Feller process with values in the Polish space $\bar \rC$ (with complete metric induced by 
the supremum norm). 

\begin{proposition} \label{tight}
For any (possibly random) $\rC$-valued initial condition $X_0$  which is nonzero almost surely, the laws $\cL(S_t),\,t \ge 2$ are tight in $\bar \rC$.
\end{proposition}
\begin{proof}
Let $\cM:=\{\cL(S_t),t \ge 2\}$. 
By the Arzel\`a-Ascoli Theorem, we have to show that 
\begin{itemize}
\item[(i)] $\lim_{a \to \infty} \sup_{\nu \in \cM} \nu (\{f \in \bar \rC:\,|f(0)|\ge a\})=0$
\item[(ii)] For every $\varepsilon >0$ we have
$$
\lim_{\delta \downarrow 0}\sup_{\nu \in \cM} \nu(\{f \in \bar \rC: \sup\{|f(t)-f(s)|: s,t \in [-1,0],\,|t-s|\le \delta \}\ge \varepsilon\})=0.                 
$$
\end{itemize}
(i) holds since $\P\{|S_t(0)|\ge a\} = \P\{|X(t)|/\Norm{X_t} \ge a\}=0$ whenever $a>1$ and $t \ge 0$. 

It remains to verify (ii). Fix $t \ge 2$. For $\alpha,\delta,\varepsilon>0$ we have
\begin{align*}
\P\Big\{&\sup_{-1 \le s \le u \le 0,u-s \le \delta} |S_t(u)-S_t(s)| \ge \varepsilon \Big\}= 
\P\Big\{\sup_{0 \le s \le u \le 1,u-s \le \delta} |H(u)-H(s)|\ge \varepsilon\Norm{X_t}\Big\} \\ 
&\le \P\Big\{\sup_{0 \le s \le u \le 1,u-s \le \delta} |H(u)-H(s)|\ge \varepsilon \alpha \|X_{t-1}\|\Big\} + \P\Big\{\Norm{X_t} \le \alpha \|X_{t-1}\|\Big\}, 
\end{align*}
where 
$$
H(r):=\int_{t-1}^{t-1+r} X(v-1)\,\dd W(v), \, r \ge 0
$$
is a local martingale which has a representation $H(r)=B(\tau(r))$ for a Brownian motion $B$ which is independent of $\cF_{t-1}$, where 
$$
\tau(r)=\int_{t-1}^{t-1+r} X^2(v-1)\,\dd v, \, r \ge 0,
$$
so $0 \le \tau'(r) = X^2(t-2+r) \le \|X_{t-1}\|^2$ for $r \in [0,1]$. Hence
\begin{align*}
\P\Big\{\sup_{0 \le s \le u \le 1,u-s \le \delta} &|H(u)-H(s)|\ge \varepsilon \alpha\|X_{t-1}\|\Big|\cF_{t-1}\Big\} \\
&\le \P\Big\{\sup_{0 \le s \le u \le 1,u-s \le \delta} \Big|\big(B(\|X_{t-1}\|^2u\big)- \big(B(\|X_{t-1}\|^2s\big)\Big| \ge \varepsilon \alpha \|X_{t-1}\|\Big|\cF_{t-1}\Big\}\\
&=\P\Big\{\sup_{0 \le s \le u \le 1,u-s \le \delta} |B(u)-B(s)| \ge \varepsilon \alpha \Big\}. 
\end{align*}
Further,
\begin{align*}
\P&\Big\{\Norm{X_t} \le \alpha \|X_{t-1}\|\Big|\cF_{t-2}\Big\} \le \P\Big\{|X(t)| \le \alpha \|X_{t-1}\|\Big|\cF_{t-2}\Big\}\\
&\le \P \Big\{|X(t)| \le \alpha^{1/3} \Norm{X_{t-1}}\Big|\cF_{t-2}\Big\} + \P \Big\{\Norm{X_{t-1}} \le \alpha^{1/3} \Norm{X_{t-2}}\Big|\cF_{t-2}\Big\}\\ 
& \hspace{.5cm} + \P \Big\{\Norm{X_{t-2}} \le \alpha^{1/3} \|X_{t-1}\|\Big|\cF_{t-2}\Big\}\\
&\le \alpha^{1/3}c_1 + \alpha^{1/3}c_1 + 4 \alpha^{2/3},
\end{align*}
where we used Lemma \ref{lemma1} (for the first two summands) and Lemma \ref{lemma2} (for the last summand) in the final step.
For fixed $\alpha, \varepsilon>0$ we obtain
\begin{align*}
\limsup_{\delta \downarrow 0}&\sup_{\nu \in \cM} \nu(\{f \in \rC: \sup\{|f(t)-f(s)|: s,t \in [-1,0],\,|t-s|\le \delta \}\ge \varepsilon\})\\
&=\limsup_{\delta \downarrow 0}\sup_{t\ge 2} \P\Big\{\sup_{-1 \le s \le u \le 0,u-s \le \delta} |S_t(u)-S_t(s)| \ge \varepsilon \Big\}\\
&\le 2c_1\alpha^{1/3}+4 \alpha^{2/3}.
\end{align*}
The assertion follows since $\alpha>0$ can be chosen arbitrarily small. 
\end{proof}

\begin{remark}
The proof of the previous proposition shows that tightness holds even uniformly with respect to the initial condition, i.e.~the family 
$\cL(S^{(\eta)}_t),\,t \ge 2,\,\eta \in \rC\backslash \{0\}$ is tight. Clearly, the family $\cL(S_t),\,t \ge 0$ is also tight for each fixed 
initial condition $\eta \in \rC\backslash \{0\}$ but not uniformly with respect to $\eta$.  
\end{remark}

\begin{proposition}\label{existence}
The $\bar \rC-$valued Markov process $S_t$, $t \ge 0$ has an invariant probability measure $\mu$. 
\end{proposition}
\begin{proof}
This follows from the Krylov-Bogoliubov theorem (see \cite{DZ96}, Theorem 3.1.1) by Proposition \ref{tight} and 
the fact that the process $S_t$, $t \ge 0$ is Feller.
\end{proof}


For later use, we formulate the following straightforward corollary of Lemma \ref{lemma1} and Lemma \ref{lemma2}.

\begin{corollary}\label{coro1}
For $\gamma>0$ and $t \ge 1$, we have
$$
\P \big( |X(t)| \le \gamma \|X_t\| \big| \cF_{t-1}\big) \le c_1 \sqrt{\gamma} + 4 \gamma \mbox{ a.s.}. 
$$
\end{corollary} 

The previous corollary immediately implies the following one.

\begin{corollary}\label{coro2} There exists some $c_2 \in (0,1)$ such that for every $t \ge 1$, we have 
$$
\P\{\bar W^*(\|X_t\|_2^2)\le \frac 13 |X(t)|\big|\F_{t-1}\} \ge c_2 \mbox{ a.s.}, 
$$
where $\bar W$ is a Wiener process which is independent of $\F_t$ and $\bar W^*(s):=\sup_{u \in [0,s]}|\bar W(u)|$.
\end{corollary}

\section{Uniqueness of an invariant measure}

Consider
\begin{align}\label{pair}
\left\{
\begin{array}{rll}
\dd X(t) &=& X(t-1)\,\dd W(t)\\
\dd Y(t) &=& Y(t-1)\,\dd W(t) + \lambda \rho (t) (X(t)-Y(t))\,\dd t,
\end{array}
\right.
\end{align}
where $\rho$ is an adapted process taking values in $\{0,1\}$ such that $\rho$ is constant on each interval $[n,n+1)$, $n \in \N_0$. 
We will show that $\rho$ can be defined in such a way that for any pair of deterministic initial conditions $(X_0,Y_0)$, the process $Z(t):=X(t)-Y(t)$ satisfies 
$\lim_{\lambda \to \infty}\limsup_{t \to \infty}\frac 1t \log \|Z_t\| =-\infty$ almost surely and such that the law of $Y$ is absolutely continuous 
with respect to the law of the solution of  $\dd \bar Y(t) = \bar Y(t-1)\,\dd W(t)$ with the same initial condition as $Y$ 
provided that $\lambda$ is sufficiently large. Then, we project both $X$ and $Y$ to  the unit sphere $\bar \rC$ and show that 
the distance between the projected processes converges to 0 as $t \to \infty$ for large enough $\lambda$. Then we apply Corollary 2.2 in 
\cite{HMS11} and obtain uniqueness.

Observe that the choice $\rho\equiv 1$ in \eqref{pair} will not work: $Y$ will not be absolutely continuous with respect to  
$\bar Y$ since it can happen that at some (random) time $Y(t)$ is zero and $Z(t)$ is not  and then the additional drift prevents 
$Y$ from being absolutely continuous with respect to 
$\bar Y$. To prevent this, we will switch off $\rho$ when this happens. Roughly speaking, we will switch $\rho$ on as often as possible (thus 
guaranteeing that $Z$ converges to 0 sufficiently quickly) but we will switch $\rho$ off whenever $Y$ has not been bounded away from zero 
sufficiently during the past unit time interval. We will always assume that the initial conditions $X_0$ and $Y_0$ are almost surely different 
which implies that the process $Z_t$ will almost surely never hit zero.

To define $\rho$, let 
$$
B:=\{f \in \rC: \,f(s)\neq 0 \mbox{ for all } s \in [-1,0] \mbox{ and } \inf_{s \in [-1,0]}|f(s)| \ge \frac 12 \sup_{s \in [-1,0]}|f(s)|\}.
$$
Further, let $\kappa>0$ be such that $c_1 \sqrt{\kappa} + 4 \kappa \le \frac{c_2}2$ (where $c_1$ and $c_2$ were defined in  
Lemma \ref{lemma1} and Corollary \ref{coro2} respectively) and define 
$$
R:=\{f \in \rC:\kappa \|f\| \le |f(0)|\}
$$
($R$ stands for {\em reasonable}) and
$$
A_n:=\{Y_n \in B\}\cap\{Z_n \in R\},\quad n \in \N_0.
$$
We define $\rho(t)=1$ on $[n,n+1)$ if $A_n$ occurs and $\rho(t)=0$ otherwise. The following lemma shows that the conditional laws of the
waiting times between successive $A_n$'s have a geometric tail (uniformly in $\lambda$).

%

\begin{lemma}\label{A}
For all $n\in \N$, $n\ge 2$, and all $\lambda \ge 0$, 
$$
\P (A_n \cup A_{n-1}|\cF_{n-2})\ge \frac{c_2}2 \mbox{ on } A_{n-2}^c \mbox{ a.s. }
$$
\end{lemma}

\begin{proof}
On the set $A_{n-1}^c$, we have
\begin{align*}
\P\big\{Y_n &\in B\big|\F_{n-1}\} \ge \P\big\{\sup_{s \in [n-1,n]} |\int_{n-1}^s  Y(u-1)\,\dd W(u)| \le  \frac 13 |Y(n-1)|\big|\F_{n-1}\big\}\\
&= \P\big\{\bar W^*(\|Y_{n-1}\|_2^2) \le \frac 13 |Y(n-1)|\big|\F_{n-1}\big\}, 
\end{align*}
where $\bar W$ is a Wiener process which is independent of $\F_{n-1}$ and $\bar W^*(t):=\sup_{s \in [0,t]}|\bar W(s)|$.
Corollary \ref{coro2} shows that  
$$
\P\{\bar W^*(\|X_{n-1}\|_2^2) \le \frac 13 |X(n-1)|\big|\F_{n-2}\} \ge c_2 \mbox{ a.s.}. 
$$
Therefore, on $A_{n-2}^c$, we have
\begin{align*}
\P\big( &\{Y_n \in B\}\cup A_{n-1}|\F_{n-2}\big)\\                
&=\P\big( \{Y_n \in B\} \cap A_{n-1}^c | \F_{n-2}\big) + \P\big( A_{n-1}|\F_{n-2})\\
&=\E\Big(\P\big(\{Y_n \in B\}|\F_{n-1}\big)\1_{A_{n-1}^c}|\F_{n-2}\Big)+ \P\big( A_{n-1}|\F_{n-2})\\
&\ge  c_2 \P\big( A_{n-1}^c|\F_{n-2}) + \P\big( A_{n-1}|\F_{n-2})\\
&\ge c_2.
\end{align*}
Further, on $A_{n-1}^c$, by Corollary \ref{coro1},
\begin{align*}
\P\big(\{Z_n &\in R\}\big|\F_{n-1}\big)=\P\big(\{Z(n) \ge \kappa \|Z_n\|\}\big|\F_{n-1}\big)\\
&\ge 1-c_1\sqrt{\kappa} - 4 \kappa \ge 1-\frac{c_2}2.
\end{align*}
Hence, on $A_{n-2}^c$, we have 
\begin{align*}
\P\big( &\{Z_n \in R\}\cup A_{n-1}|\F_{n-2}\big)\\   
&=\P\big( \{Z_n \in R\} \cap A_{n-1}^c | \F_{n-2}\big) + \P\big( A_{n-1}|\F_{n-2})\\
&=\E\Big(\P\big(\{Z_n \in R\}|\F_{n-1}\big)\1_{A_{n-1}^c}|\F_{n-2}\Big)+ \P\big( A_{n-1}|\F_{n-2})\\
&\ge  \Big(1-\frac{c_2}{2}\Big)  \P\big( A_{n-1}^c|\F_{n-2}) + \P\big( A_{n-1}|\F_{n-2})\\
&\ge 1- \frac{c_2}{2}.
\end{align*}
Therefore, on $A_{n-2}^c$, we have
$$
\P (A_n \cup A_{n-1}|\cF_{n-2})\ge c_2 + 1-c_2/2-1=\frac{c_2}2,
$$
which is the assertion.
\end{proof}

We now have to show that whenever we have an interval $[n,n+1)$ on which $\rho$ is one, then with high probability $\Norm{Z_{n+1}}$ is much 
smaller than $\Norm{Z_n}$ (when $\lambda$ is large).
More precisely, the following lemma holds.
\begin{lemma}\label{log}
We have
$$
\E\Big(\frac {\Norm{Z_{n+1}}} {\Norm{Z_n}} \big| \F_n\Big) \le 2(r(\lambda))^{1/2}\qquad \mbox{ on } A_n,
$$
and 
$$
\E\Big(\frac {\Norm{Z_{n+1}}} {\Norm{Z_n}} \big| \F_n\Big) \le 2\sqrt 2\qquad \mbox{ on } A_n^c,
$$
for $\lambda >0$ and
$$
r(\lambda):=\frac 2{\kappa^2} \frac 1{\lambda}. 
$$
Let $\lambda_0:=\frac 2{\kappa^2}$ (which implies $r(\lambda)\le 1$ for $\lambda \ge \lambda_0$). Then, for $\alpha>0$, $n \in \N_0$, and $\lambda \ge \lambda_0$, we have
$$
\P\big\{\Norm{Z_{n+3}}   \ge \alpha \Norm{Z_n}\big|\cF_n \big\} \le 3  \alpha^{-2/3}r(\lambda)^{1/3} + \Big( 1-\frac{c_2}2\Big) \big(\big(6 \alpha^{-2/3}\big)\wedge 1\big) 
\mbox{ a.s.}.
$$
\end{lemma}

\begin{proof}
On $A_n=\{Y_n \in B\}\cap \{Z_n \in R\}$, we have for $\alpha>0$ and $\lambda>0$
\begin{align}\label{eins}
\P\{&\Norm{Z_{n+1}}\ge \alpha \Norm{Z_n}\big| \F_n\} \le \frac 1{\alpha^2\Norm{Z_n}^2} \E\big(\Norm{Z_{n+1}}^2\big|\F_n\big)\nonumber\\
&=\frac 1{\alpha^2\Norm{Z_n}^2} \Bigg(Z^2(n)\Big(\frac {1-\ee^{-2\lambda}}{2\lambda}+\ee^{-2\lambda}\Big)   
+\int_0^1 \Big(\frac {1-\ee^{-2\lambda(1-u)}}{2\lambda}+\ee^{-2\lambda (1-u)} \Big) Z^2(n-1+u)\, \dd u\Bigg)\nonumber\\
&\le \frac {Z^2(n)}{\alpha^2\Norm{Z_n}^2}r(\lambda) 
\le \frac 1 {\alpha^2}r(\lambda), 
\end{align}
since $\kappa \le 1$. Therefore, on $A_n=\{Y_n \in B\}\cap \{Z_n \in R\}$,
\begin{align*}
\E\Big(\frac {\Norm{Z_{n+1}}} {\Norm{Z_n}} \big| \F_n\Big)
= \int_0^{\infty}\P\{\Norm{Z_{n+1}}\ge \alpha \Norm{Z_n}\big| \F_n\}\,\dd \alpha 
\le \int_0^{\infty} \Big( 1 \wedge  \frac {r(\lambda)}  {\alpha^2} \Big) \,\dd \alpha
= 2 (r(\lambda))^{1/2}.  
\end{align*}
 Further, on $A_n^c=\big(\{Y_n \in B\}\cap \{Z_n \in R\}\big)^c$, we have for $\alpha>0$
\begin{align}\label{zwei}
\P\{&\Norm{Z_{n+1}}\ge \alpha \Norm{Z_n}\big| \F_n\} \le \frac 1{\alpha^2\Norm{Z_n}^2} \E\big(\Norm{Z_{n+1}}^2\big|\cF_n\big)\nonumber\\
&=\frac 1{\alpha^2\Norm{Z_n}^2} \Big(2Z^2(n)+\|Z_n\|_2^2 +\int_0^1 (1-u)Z^2(n-1+u)\,\dd u\Big) \le \frac 2{\alpha^2}.
\end{align}
Hence,
\begin{align*}
\E \Big(\frac {\Norm{Z_{n+1}}} {\Norm{Z_n}} \big| \F_n\Big) = \int_0^{\infty}\P\{\Norm{Z_{n+1}}\ge \alpha \Norm{Z_n}\big| \F_n\}\,\dd \alpha  \le 2 \sqrt{2}.
\end{align*}
It remains to prove the final assertion. 
Using \eqref{eins} and \eqref{zwei} we see that on $A_n$ we have
\begin{align*}
\P\Big\{&\frac{\Norm{Z_{n+3}}}{\Norm{Z_n}} \ge \alpha \Big| \F_n\Big\} \le \P\Big\{\frac{\Norm{Z_{n+3}}}{\Norm{Z_{n+2}}} \ge \alpha^{1/3}(r(\lambda))^{-1/6} \Big| \F_n\Big\}\\ 
&+ \P\Big\{\frac{\Norm{Z_{n+2}}}{\Norm{Z_{n+1}}} \ge \alpha^{1/3}(r(\lambda))^{-1/6}     \Big| \F_n\Big\} +
\P\Big\{\frac{\Norm{Z_{n+1}}}{\Norm{Z_n}} \ge \alpha^{1/3}r(\lambda)^{1/3} \Big| \F_n\Big\}\\
&\le 5 \alpha^{-2/3} r(\lambda)^{1/3}.
\end{align*}
Using \eqref{eins} and \eqref{zwei} we see that on $A_n^c$ we have
\begin{align*}
\P\Big\{\frac{\Norm{Z_{n+3}}}{\Norm{Z_n}} \ge \alpha \Big| \F_n\Big\} \le \sum_{i=1}^3 \P\Big\{\frac{\Norm{Z_{n+i}}}{\Norm{Z_{n+i-1}}} \ge \alpha^{1/3} \Big| \F_n\Big\} 
\le 6 \alpha^{-2/3}.
\end{align*}
Arguing the same way and using Lemma \ref{A} (which implies  $\P\big( A_n \cup A_{n-1} \cup A_{n-2}|\cF_{n-2}\big) \ge c_2/2$), we obtain 
\begin{align*}
\P\Big\{\frac{\Norm{Z_{n+3}}}{\Norm{Z_n}} \ge \alpha \Big| \F_n\Big\} \le \frac{c_2}2 5  \alpha^{-2/3}r(\lambda)^{1/3} 
+ \Big(1-\frac{c_2}2\Big)\big(\big(6 \alpha^{-2/3}\big)\wedge 1\big),
\end{align*}
which implies the assertion.
\end{proof}

\begin{lemma}\label{limsup}
For $\lambda>0$, we have
$$
\limsup_{t \to \infty} \frac 1t \log |Z(t)| \le \frac {c_2}{12} \log r(\lambda) + \log 2 \quad {a.s.}.
$$
\end{lemma}

\begin{proof}
Define
$$
V_n:=\log \frac {\Norm{Z_{n+1}}} {\Norm{Z_n}},\,n \in \N_0, 
$$
and 
$$
U_n:={\bf 1}_{A_n}, \, n \in \N_0.
$$
Then, by Jensen's inequality and Lemma \ref{log}, we get 
$$
\E(V_n |\F_{n})\le U_{n} \frac 12 \log r(\lambda)  +  (1-U_{n})\frac 12 \log 2 + \log 2,
$$
so 
$$
\sum_{i=0}^N \Big(V_i-\frac 12 \big(U_{i} \log r(\lambda)  +  (1-U_{i})\log2 + \log 2\big)\Big)
$$
is a supermartingale. Due to \eqref{eins} and \eqref{zwei}, the strong law of large numbers for martingales (\cite{HH80}, Theorem 2.19) 
implies 
$$
\limsup_{N \to \infty} \frac 1N \sum_{i=0}^N \Big(V_i-\frac 12 \big(U_{i} \log r(\lambda)  +  (1-U_{i})\log2 +\log 2\big) \Big) \le 0.
$$
Hence,
$$
\limsup_{N \to \infty} \frac 1N \sum_{i=0}^N V_i \le \limsup_{N \to \infty} \frac 1{2N} \sum_{i=0}^N \big(U_{i} \log r(\lambda)  + (1-U_{i}) \log 2 + \log 2\big)
$$ 
which, using Lemma \ref{A}, is at most $\frac 12\big(\frac {c_2}6 \log r(\lambda)+ 2\log 2\big)$ almost surely. Hence,
$$
\limsup_{n \to \infty} \frac 1n \log \Norm{Z_n} \le  \frac {c_2}{12} \log r(\lambda) + \log 2  \quad {a.s.}.
$$
To obtain the assertion, it suffices to show (thanks to the first Borel-Cantelli Lemma)  
that for each $\delta>0$ the sum over $\P\{\|Z_{n+1}\| \ge \ee^{\delta n}\Norm{Z_n}\}$ is finite which is easily established by estimating the 
corresponding conditional probabilities (conditioned on 
$\F_n$) like in the proof of Lemma \ref{log}. This proves the assertion. 
\end{proof}

\begin{lemma}\label{absolute}
There exists some $\lambda_1 >0$ such that for all $\lambda \ge \lambda_1$, the law of the process $Y$ is absolutely continuous with respect to that of the solution 
of \eqref{sde} with the same initial condition as $Y$ ($\lambda_1$ does not depend on the initial condition of $(X_0,Y_0)$).
\end{lemma}

\begin{proof}
We need to make sure that for $\lambda$ sufficiently large, we have
\begin{equation}\label{toShow}
\int_0^{\infty} \rho(t) \lambda^2 \frac{Z^2(t)}{Y^2(t-1)}\,\dd t < \infty \qquad \mbox{a.s.}
\end{equation}
Then, the assertion follows from Girsanov's Theorem (see \cite{LS77}, Chapter 7).
By the definition of $\rho$, we have $Y^2(t)\ge \frac 14 \|Y_n\|^2$ whenever $t \in [n-1,n]$ and $\rho(n)=1$ (which is equivalent to 
$\rho(s)=1$ for all $s \in [n,n+1)$) which implies
\begin{align}\label{Girsi}
\int_0^{\infty} \rho(t)  \frac{Z^2(t)}{Y^2(t-1)}\,\dd t &\le 4 \sum_{n=1}^{\infty} \Big(\frac{\|Z_n\|_2}{\|Y_{n-1}\|}\Big)^2 
\end{align}
Let $\Lambda_0$ be as in Proposition \ref{aop}. Then $\liminf_{t \to \infty}\frac 1t \log \|X_t\|\ge \Lambda_0$ almost surely for each initial 
condition $\eta \neq 0$. By Lemma \ref{limsup} we find $\lambda_1 > 0$ such that for all $\lambda \ge \lambda_1$, we have 
$\limsup_{t \to \infty}\frac 1t \log |Z(t)|\le (2 \Lambda_0) \wedge (-1)$. Then 
$\liminf_{t \to \infty}\frac 1t \log \|Y_t\|\ge \Lambda_0$ which, together with equation \eqref{Girsi}, implies \eqref{toShow}.
\end{proof}

\begin{proposition}\label{support}
The Markov process  $S_t:=X_t/\Norm{X_t}$, $t\ge 0$ has a unique invariant probability measure $\mu$. The support of $\mu$ is $\bar C$.
\end{proposition} 

\begin{proof}
Existence of an invariant probability measure $\mu$ has been shown in Proposition \ref{existence}. To establish uniqueness, observe that
\begin{equation}\label{estimat}
\BigNorm{\frac{X_t}{\Norm{X_t}}-\frac{Y_t}{\Norm{Y_t}}}\le \BigNorm{\frac{X_t-Y_t}{\Norm{X_t}}+Y_t\Big(\frac 1{\Norm{X_t}}-\frac 1{\Norm{Y_t}}\Big)} \le
2\frac{\Norm{Z_t}}{\Norm{X_t}},
\end{equation}
which converges to zero exponentially fast as long as $\lambda$ is sufficiently large.   Lemma \ref{absolute}, together with the fact that 
absolute continuity of measures is preserved under measurable maps, shows that the law $\cL (Y_t/\Norm{Y_t},t\ge 0)$ 
is absolutely continuous with respect to $\cL (\bar Y_t/\Norm{\bar Y_t},t\ge 0)$, where $\bar Y$ solves equation \eqref{sde} with the 
same initial condition as $Y$. Now uniqueness follows from Corollary 2.2 in \cite{HMS11}. 

It remains to show that $\mu$ has full support. Let $X$ solve \eqref{sde} with initial distribution $\mu$. Let $G$ be a non-empty open subset of 
$\bar \rC$. We show that $\mu(G)>0$. Assume that $G$ contains a function $f$ such that $f(0)>0$ (otherwise the proof is completely analogous). 
Let $B^+$ be the set of positive functions in $B$. 
It follows as in Lemma \ref{A} that $X_n$ visits $B^+$ 
infinitely often almost surely. If $X_n \in B^+$, then $\P\big\{S_{n+1} \in G \big|\cF_n\big\}>0$ and therefore $\mu (G) >0$.  
\end{proof}

\section{Proof of Theorem \ref{main}}

In order to complete the proof of Theorem \ref{main}, we need to show that \eqref{Grenzwert} does not only hold for $\mu$-almost every initial condition 
but for every initial condition in $\bar \rC$. To establish this, we prove the following lemma.

\begin{lemma}\label{wichtig}
There exists some $\lambda_2>0$ such that for each $\phi \in \bar \rC$ and each $\lambda \ge \lambda_2$ the following holds. 
Let $\eta \in \bar \rC$ and let $(X,Y)$ solve  \eqref{pair} with initial condition $(\eta,\phi)$. Define $\rho$ and $Z$ as before. Then
$$
\int_0^{\infty} \rho(t) \frac{Z^2(t)}{Y^2(t-1)}\,\dd t \to 0 \mbox{ in probability as } \|\eta-\phi\| \to 0.
$$
\end{lemma}

\begin{proof}
As in the proof of Lemma \ref{absolute}, we get
\begin{equation}\label{summe}
\int_0^{\infty} \rho(t) \frac{Z^2(t)}{Y^2(t-1)}\,\dd t \le 4 \sum_{n=1}^{\infty} \Big(\frac{\Norm{Z_n}}{\|Y_{n-1}\|}\Big)^2 
\le 16 \sum_{n=1}^{\infty} \Big(\frac{\Norm{Z_n}}{\Norm{Y_{n-1}}}\Big)^2 .
\end{equation}
First, we estimate the numerator in the sum from above.
For $\lambda \ge \lambda_0$ (defined in Lemma \ref{log}), let $U_{\lambda}$ be a random variable satisfying 
$$
\P\{U_{\lambda} \ge \alpha\}= \Bigg(3  \alpha^{-2/3}r(\lambda)^{1/3} + \Big( 1-\frac{c_2}2\Big) \big(\big(6 \alpha^{-2/3}\big)\wedge 1\big)\Bigg)\wedge 1,
$$
for $\alpha \ge 0$. Note that $\E U_{\lambda}^{1/3} \le \E U_{\lambda_0}^{1/3} <\infty$. 
Define $c_3:=1-c_2/4$.
For each $\delta>0$, we can find some $\gamma_0 \in (0,1/3]$ and some $\lambda_2\ge \lambda_0$ for which $\E U_{\lambda_2}^{\gamma_0} \le c_3 \exp\{-3\gamma_0 \delta\}$. 
For $m \in \N$, let $\Gamma_m:=\log\big(\Norm{Z_{3m}}/\Norm{Z_{3(m-1)}}\big)$. Then Markov's inequality and the last part of Lemma \ref{log} imply for $C>0$:
\begin{align*}
\P\big\{\big( \Norm{Z_{3k}}\ee^{3k\delta }\big) \ge C\big\} &= \P\Big\{\sum_{m=1}^k \Gamma_m +\log \Norm{Z_0} \ge \log C -3k \delta\Big\}\\
&\le C^{-\gamma_0}\Norm{Z_0}^{\gamma_0} \ee^{3\gamma_0 k \delta} \E \Big(\prod_{m=1}^k \E\big(\ee^{\gamma_0 \Gamma_m} \big| \cF_{3(m-1)}\big)\Big)\\
&\le C^{-\gamma_0}\Norm{Z_0}^{\gamma_0} \ee^{3\gamma_0 k \delta} \Big(\E \big(   U_{\lambda_2}^{\gamma_0}\big)\Big)^k\\
&\le  C^{-\gamma_0}\Norm{Z_0}^{\gamma_0} c_3^k. 
\end{align*}
For $i=1,2$, we obtain in the same way
$$
\P\big\{\big( \Norm{Z_{3k+i}}\ee^{3k\delta }\big) \ge C \big\} \le  C^{-\gamma_0}\E \Norm{Z_i}^{\gamma_0} c_3^k \le  C^{-\gamma_0}2^{\gamma_0}\Norm{Z_0}^{\gamma_0} c_3^k .  
$$
Hence,
\begin{align}\label{absch}
\P\big\{ \sup_{k \in \N_0}\big(\Norm{Z_{k}}\ee^{k\delta }\big) > C\big\} &\le \sum_{k=0}^{\infty} \P\big\{\Norm{Z_{k}}\ee^{k\delta } \ge C\big\}\nonumber\\
&\le   C^{-\gamma_0} (2\ee^{2\delta})^{\gamma_0} c_3^{-2/3} \Norm{Z_0}^{\gamma_0} (1-c_3^{1/3})^{-1}.
\end{align}

Now, we estimate the denominator in \eqref{summe}. Observe that for $A,\beta>0$ 
\begin{align}\label{zw}
\P\{\Norm{Y_m} \le A\ee^{-\beta m}\}&\le \P\{\Norm{X_m} \le \frac 32 A\ee^{-\beta m}\}+ \P\{\Norm{Z_m} \ge \frac A2 \ee^{-\beta m}\}. 
\end{align}
Using Proposition \ref{aop1}, we get
\begin{align}\label{weiter}
\P\Big\{\Norm{X_m} \le& \frac 32 A\ee^{-\beta m}\Big\}
=\P\Big\{\Norm{X_m}^{-1/2} \ge \big(\frac 32 A\ee^{-\beta m}\big)^{-1/2}\Big\}\nonumber\\
&\le K^m \Norm{\eta}^{-1/2} \big( \frac 32 A\ee^{-\beta m}\big)^{1/2}=K^m \big( \frac 32 A\ee^{-\beta m}\big)^{1/2}
\end{align}
which decays exponentially fast provided that $\beta$ is sufficiently large. Fix such a $\beta>0$ and let $\delta:=2 \beta$. Using \eqref{absch}, \eqref{zw}, and 
\eqref{weiter}, we get
\begin{align}\label{auch}
\P\big\{ \inf_{m \in \N_0}\big(\Norm{Y_{m}}\ee^{m\beta }\big) < A\big\} \le \sum_{m=0}^{\infty} \P\big\{\Norm{Y_{m}}\ee^{m\beta } \le A\big\}
\le c_4 A^{1/2} + c_5 A^{-\gamma_0}\Norm{Z_0}^{\gamma_0},
\end{align}
where $c_4,c_5$ are constants which depend on $\beta$ and $\gamma_0$ (which are fixed) but not on $A$. Choosing $A$ sufficiently small 
and $C$ even smaller, we see from \eqref{absch} and \eqref{auch} that for $\Norm{Z_0}$ small enough the right-hand side of \eqref{summe} 
is as small as we like with a probability as close to one as we like. This proves the lemma.
\end{proof}

\begin{proofone} We have established existence and uniqueness of an invariant probability measure $\mu$ of the Markov process $S_t$, $t \ge 0$ on $\bar \rC$ in 
Proposition \ref{support}. 
Let $f,g$ and $\Lambda$ be as defined in \eqref{Grenzwert}. 
It remains to show that for each initial condition $\eta \in \rC\backslash\{0\}$ (or $\eta \in \bar \rC$) the solution $X$ of equation \eqref{sde} satisfies \eqref{Grenzwert}. 
Let $\cM \subseteq \bar \rC$ be the set of initial conditions for which the empirical distribution of $S_t$, $t \ge 0$ converges to $\mu$ weakly almost surely {\em and} 
for which $\lim_{t \to \infty} \frac 1t \int_0^t g^2(S_s)\, \dd s=\int g^2\,\dd \mu$ holds almost surely (the second condition does not follow from the first since $g^2$ is unbounded 
but  $\lim_{t \to \infty} \frac 1t \int_0^t f(S_s)\, \dd s=\int f\,\dd \mu$ does since $f$ is bounded and continuous). 
Once we have shown that $\cM=\bar \rC$ then Theorem \ref{main} follows.\\ 

\noindent {\bf Step 1:} We show that there exists some $\lambda_3 >0$ such that for each pair $(\eta  ,\phi )$ of distinct non-zero initial conditions, the solution $(X,Y)$  of 
\eqref{pair} with $\lambda \ge \lambda_3$ satisfies
\begin{equation}\label{g}
\lim_{s \to \infty}\Big(\frac{Y(s-1)Y(s)}{\Norm{Y_s}^2}\Big)^2-\Big(\frac{X(s-1)X(s)}{\Norm{X_s}^2}\Big)^2 = 0 \mbox{ a.s.},
\end{equation}
(i.e.~$\lim_{s \to \infty}\big( g^2(Y_s/\Norm{Y_s})-g^2(X_s/\Norm{X_s})\big) =0$) and
\begin{equation}\label{g2}
\lim_{s \to \infty}\Big\|\frac{Y_s}{\Norm{Y_s}}-\frac{X_s}{\Norm{X_s}}\Big\| = 0 \mbox{ a.s.}.
\end{equation}
Replacing  $Y$ by $X-Z$, we get
\begin{align*}
\Big(&\frac{Y(s-1)Y(s)}{\Norm{Y_s}^2}\Big)^2-\Big(\frac{X(s-1)X(s)}{\Norm{X_s}^2}\Big)^2 \\
&= \frac{(X(s-1)X(s))^2\big(\Norm{X_s}^4-\Norm{Y_s}^4\big) + A(s)   }{\Norm{Y_s}^4\Norm{X_s}^4},
\end{align*}
where $A(s)$ is a polynomial of degree 8 of the variables $\Norm{X_s}$, $\Norm{Y_s}$, $X(s)$, $X(s-1)$, $Z(s)$, and $Z(s-1)$ such that each 
summand contains either $Z(s)$ of $Z(s-1)$ at least once. Choosing $\lambda$ sufficiently large, $Z$ decays to zero with an exponential rate as large as we desire by Lemma \ref{limsup}. 
Since we also have \`a priori upper and lower bounds for the exponential decay of $X$ (and hence of $Y$), we see that $\lim_{s \to \infty} A(s)/(\Norm{Y_s}^4\Norm{X_s}^4)=0$ for 
large enough $\lambda$. The same is true for the remaining term: just apply the formula $a^4-b^4=(a-b)(a^3+a^2b+ab^2+b^3)$  with $a=\Norm{X_s}$ and $b=\Norm{Y_s}$. Clearly, 
\eqref{g2} also holds for sufficiently large $\lambda$ (cf. \eqref{estimat} with the outer $\Norm{.}$ replaced by the sup-norm).
Therefore, there exists some $\lambda_3 >0$ such that \eqref{g} and \eqref{g2} hold for every $\lambda \ge \lambda_3$.\\

\noindent {\bf Step 2:} Fix an initial condition $\phi \in \bar \rC$ and denote the solution of \eqref{sde} with initial condition $\phi$ by $\bar Y$. We will show that $\phi \in \cM$. 
Let $\lambda \ge \lambda_3$ with $\lambda_3$ as defined in the first step. 
We know that $\mu(\cM)=1$ by Birkhoff's ergodic theorem and the fact that $g^2$ is $\mu$-integrable (cf. the statement after Proposition \ref{aop}). 
From Proposition \ref{support} we know that the support of $\mu$ is $\bar \rC$, so $\cM$ is dense in $\bar \rC$. 
For a given $\lambda \ge \lambda_3$ and $\varepsilon>0$, applying Lemma \ref{wichtig}, we can find some $\eta \in  \cM$ such that for 
$$
V:= \int_0^{\infty} \big(v(s)\big)^2\,\dd s, \mbox{ where } v(s):= \frac{\lambda \rho(s) Z(s)}{Y(s-1)},
$$
we have $\P\{V < 1\} \ge 1-\varepsilon$. 
Define the stopping time 
$$
\tau:=\inf\Big\{ u \ge 0:  \int_0^u v^2(s)\,\dd s \ge 1\Big\},
$$ 
and let $\tilde Y$ solve
\begin{equation}\label{tilde}
\dd \tilde Y(t)=\tilde Y (t-1)\,\dd \tilde W(t),\qquad \tilde Y_0=\phi, 
\end{equation}
where
$$
\tilde W(t):=W(t)+\int_0^{t \wedge \tau} v(s)   \dd s. 
$$
By the Cameron-Martin-Girsanov Theorem, $\tilde W$ is a Wiener process with respect to the measure $\tilde \P$ defined as $\dd\tilde \P(\omega)= U(\omega)  \dd \P(\omega)$,
where 
$$
U:=\exp\Big\{ -\int_0^{\tau}v(s)\dd W(s) - \frac 12 \int_0^{\tau}v^2(s)\dd s\Big\}.
$$
By uniqueness of solutions of \eqref{tilde}, $Y$ and $\tilde Y$ agree almost surely up to $\tau$. In particular, $\P\{Y\equiv \tilde Y\}\ge 1-\varepsilon$. 
Let 
$\Gamma \in \cF$ denote the set of all $\omega$ for which the empirical distribution of $\bar Y_t/\Norm{\bar Y_t}$, $t > 0$ converges to $\mu$ weakly and the 
corresponding integrals of $g^2$ converge as well. We want to show that $\P(\Gamma)=1$ (which is equivalent to $\phi \in \cM$). 
Let $\tilde \Gamma$ be the subset of those $h \in \rC[-1,\infty)$ for which the 
empirical distribution $t^{-1}\int_0^t \delta_{h_s} \dd s$ of $h$ converges to $\mu$ weakly 
as $t \to \infty$ and the corresponding integrals of $g^2$ converge as well. 
We have
\begin{align*}
\P(\Gamma^c)&=\P\{\bar Y \in \tilde \Gamma^c\} = \tilde \P\{\tilde Y \in \tilde \Gamma^c\} = \int {\bf 1}_{\{\tilde Y \notin \tilde \Gamma\}}\frac{\dd \tilde \P}{\dd \P} \,\dd \P
=\E \big({\bf 1}_{\{\tilde Y \notin \tilde \Gamma\}} U \big)\\
&\le \big( \P\{\tilde Y \notin \tilde \Gamma\}\big)^{1/2} \big(\E U^2\big)^{1/2}
\le \big( \P\{Y \notin \tilde \Gamma\}+ \varepsilon \big)^{1/2} \big(\E U^2\big)^{1/2}=\varepsilon^{1/2} \big(\E U^2\big)^{1/2},
\end{align*}
where $ \P\{Y \notin \tilde \Gamma\}=0$ follows from Step 1.
The second moment of $U$ is easily seen to be bounded by a universal constant.  Since $\varepsilon>0$ was arbitrary, we get $\P(\Gamma)=1$, so the assertion of Step 2 
follows and the proof of Theorem \ref{main} is complete. 
\end{proofone}


\begin{thebibliography}{10}%
\bibliographystyle{plain}

\bibitem{AKO84} Arnold. L., Kliemann, W., and Oeljeklaus, E. (1984). Lyapunov exponents for linear stochastic systems. 
In: Lyapunov Exponents (ed: Arnold, Wihstutz), Springer LNM 1186, pp. 85--128, Springer, Berlin. 

\bibitem{DZ96} Da Prato, G. and Zabczyk, J. (1996). {\em Ergodicity for Infinite Dimensional Systems}, Cambridge University Press, Cambridge.

\bibitem{F63} Furstenberg, H. (1963). Noncommuting random products, {\em Trans. Amer. Math. Soc.} {\bf 108}, 377--428.

\bibitem{HMS11} Hairer, M., Mattingly, J., and Scheutzow, M. (2011). Asymptotic coupling and a general form of Harris' 
theorem with applications to stochastic delay equations, {\em Prob. Theory Rel. Fields} {\bf 149}, 223--259. 

\bibitem{HH80} Hall, P. and Heyde, C. (1980). {\em Martingale Limit Theory and its Application}, Academic Press, New York.

\bibitem{H67} Hasʹminskii, R. Z. (1967). Necessary and sufficient conditions for asymptotic stability of linear stochastic systems, 
{\em Theory Probability Appl.} {\bf 12}, 144–-147.

\bibitem{LS77} Liptser, R.S. and Shiryayev, A.N. (1977). {\em Statistics of Random Processes I, General Theory}, Springer, New York.

\bibitem{M86} Mohammed, S.~(1986). Nonlinear flows of stochastic linear delay equations, {\em Stochastics} {\bf 17}, 207-–213. 

\bibitem{MS96} Mohammed, S.~and Scheutzow, M.~(1996). Lyapunov exponents of linear stochastic functional differential equations driven 
by semimartingales. I. The multiplicative ergodic theory, {\em Ann.~Inst.~H.~Poincar\'e Probab.~Statist.} {\bf 32}, 69-–105. 

\bibitem{MS97} Mohammed, S.~and Scheutzow, M.~(1997). Lyapunov exponents of linear stochastic functional differential equations driven 
by semimartingales. II. Examples and case studies, {\em Ann.~Probab.} {\bf 25}, 1210--1240. 

\bibitem{Sch05} Scheutzow, M. (2005). Exponential growth rates for stochastic delay equations,
{\em Stoch.~Dyn.}  {\bf 5}, 163--174.
\end{thebibliography}
\end{document}